\documentclass{amsart}
\usepackage{amsmath}
\usepackage{paralist}
\usepackage{graphics} 
\usepackage{epsfig} 
\usepackage{graphicx}  
\usepackage{epstopdf}
\usepackage[colorlinks=true]{hyperref}
\hypersetup{urlcolor=blue, citecolor=red}

\newtheorem{theorem}{Theorem}[section]

\newtheorem{lemma}[theorem]{Lemma}
\newtheorem{proposition}{Proposition}

\theoremstyle{definition}

\newtheorem{remark}{Remark}

\newcommand{\R}{\mathbb{R}}
\newcommand{\N}{\mathbb{N}}
\DeclareMathOperator{\spt}{supp}
\DeclareMathOperator{\diam}{diam}
\DeclareMathOperator{\dive}{div}
\DeclareMathOperator{\curl}{curl}

\title[{Classical solutions for the system ${\text{curl}\,v=g}$}] 
{Classical solutions for the system $\mathbf
  {\text{curl}\,v=g}$, with vanishing
  Dirichlet boundary conditions}

\author[Luigi C. Berselli and Placido Longo]{}


\email{luigi.carlo.berselli@unipi.it}
\email{placido.longo@unipi.it}

\begin{document}

\maketitle

\centerline{\scshape Luigi C. Berselli and Placido Longo}
\medskip
{\footnotesize
 \centerline{Dipartimento di Matematica}
\centerline{Universit\`a di Pisa}
   \centerline{Via F. Buonarroti 1/c}
   \centerline{Pisa, ITALY, I-56127}
} 

\bigskip


\begin{abstract}
  We consider the boundary value problem associated to the curl operator, with vanishing
  Dirichlet boundary conditions. We prove, under mild regularity of the data of the
  problem, existence of classical solutions.

\bigskip

\end{abstract}

\section{Introduction}
The aim of this paper is to solve, in the classical setting, the first-order boundary
value problem
\begin{equation}
  \label{eq:curl}
  \begin{aligned}
    \curl v&=g\qquad\text{in }\Omega,
    \\
    v&=0\qquad\text{on }\partial\Omega.
  \end{aligned}
\end{equation}
More precisely, we look for solutions which are vector fields $v\in (C^{1}(\Omega)\cap
C^{0}(\overline{\Omega}))^3$, where $\Omega\subset\R^{3}$ is a smooth and bounded
domain. The history of this problem is rather long and strictly connected with that of the
solution of the divergence equation, with the Helmholtz decomposition, and with the
analysis of certain topological properties of the domain. We observe that if such a
solution exists and is also two times differentiable (at least in a weak sense), then
necessarily $\dive g=\dive\curl v=0$. This gives a natural compatibility condition for the
datum of the problem. In addition, if the problem is studied in the whole space, or in the
periodic setting (or in general when no condition is imposed at the boundary), then
constructing a solution is very simple. We nevertheless recall that in general may exist
infinite solutions to~\eqref{eq:curl} and in practical problems, constructing at least one
with the requested regularity is enough. By following the classical work of von Helmholtz
in electromagnetism~\cite{Hel1870}, we can directly check that if for $x\in \R^{3}$ we
define the vector $G(x)$ by means of the solution in the whole space of $\Delta G(x)=g(x)$
namely by using the Newtonian potential
$G^{j}(x)=-\frac{1}{4\pi}\int_{\R^{3}}\frac{g^{j}(y)}{|x-y|}\,dy$, and if we set
$A:=-\curl G$, then (provided that $\dive g=0$)
\begin{equation*}
\curl A=g.
\end{equation*}
It was later recognized the extremely high relevance of this result, when formulated in
appropriate function spaces, in the theory of partial differential equations, especially
in connection with classical electromagnetism and fluid mechanics.  A systematic study of
space decomposition in the sum of a gradient and the curl of a vector potential was
initiated by Weyl~\cite{Wey1940}. The reader can find an up-to-date reference in
Galdi~\cite[Ch. III]{Gal2011}, with the relevant applications to the field of mathematical
fluid mechanics.

On the other hand, the solution of the various curl systems in presence of boundary
conditions requires a special treatment, since the problems may impose some geometric
conditions on the domain and the construction requires special techniques. In this respect
we recall especially the papers by Borchers and Sohr~\cite{BS1990}, von
Wahl~\cite{vWa1990}, and Bolik and von Wahl~\cite{BvW1997}, based on representation
formulas and on reduction to appropriate integral equations. We also recall the recent
results by Kozono and Yanagisawa~\cite{KY2009}, based on the classical theory of Agmon,
Douglis, and Nirenberg and to reduction to a suitable family of boundary value problems.

In this paper we are mainly interested in classical solutions: to find an appropriate $X
\subseteq (C(\Omega))^{3}$ such that, for all $g\in X$, one can construct a solution $v$
of class $C^{1}$ in the open set $\Omega$, satisfying the homogeneous boundary conditions
in the classical sense of $C(\overline{\Omega})$. We work in domains star-shaped with
respect to a ball, hence the topology of the domain is very simple, avoiding the
pathologies considered in~\cite{BvW1997}. As the main point of our paper is to find a
suitable subspace of continuous functions, such that the problem can be solved, this is
connected with the limiting theory (that is in $L^{\infty}(\Omega)$) for partial
differential equations. We recall that the solution of elliptic equations with right-hand
side in $L^{\infty}(\Omega)$ does not produce --even for the Poisson problem-- functions
with bounded second order derivatives. Nevertheless, it is well-known that for
$2^{\text{nd}}$ order elliptic equations by assuming a H\"older (denoted by
$C^{0,\alpha}(\Omega)$) or even Dini (denoted by $C_{D}(\Omega)$) continuous datum is
enough to produce $C^2(\Omega)$ solutions. Here, the situation is even more complex, due
to the fact that~\eqref{eq:curl} is not an elliptic systems. In addition, we are not
interested in special features as those exploited by Bourgain and Brezis~\cite{BB2007} for
the right-hand side in $L^{3}(\Omega)$ concerning whether or note there exists at least a
solution in $W^{1,3}(\Omega)\cap L^{\infty}(\Omega)$.

Our aim is to exploit a representation formula similar to the one introduced by
Bogovski\u\i~\cite{Bog1980} for the divergence system (see also the review in
Galdi~\cite{Gal2011}), for which we have recently proved in~\cite{BL2017} a similar
results of resolvability in the space of classical solutions.  For our purposes is then
relevant to use an integral representation formula, which is on the same lines of that
previously introduced by Griesinger~\cite{Gri1990a,Gri1990b}. We use classical tools as
those developed by Korn~\cite{KR1} for the study of H\"older continuity (and explained for
the Newtonian potential in Gilbarg and Trudinger~\cite{GT1998}) and we use fine properties
of the Dini continuous functions to determine a class of right-hand sides producing a
$C^{1}(\Omega)$ solution.  The introduction of the integral control of the modulus of
continuity dates back to Dini~\cite{Dini1902} for elliptic equations. Starting from this
work it became a sort of classical borderline conditions to have continuity of second
order derivatives for elliptic equations of second order. We wish to mention that the
assumption of Dini continuous data in problems of  fluid mechanics started with the
paper of Beir\~ao da Veiga~\cite{Bei1984}. Therein, considering the 2D Euler equations for
incompressible fluids, a unique solution is constructed in the critical space for the
vorticity $C(0,T;C_{D}(\Omega))$. More recently the same functional setting have been also
employed by Koch~\cite{Koch2002} and in~\cite{BB2015} to analyze properties of the
long-time behavior. In addition, the construction of classical solutions of the Stokes
system with Dini-continuous data has been recently provided in papers by Beir\~ao da
Veiga~\cite{Bei2014,Bei2015,Bei2016} and this motivates also our analysis of the
divergence and curl operator, since they are some of the building blocks in the theory of
mathematical fluid mechanics.

We already observed that if $\curl v=g$, then necessarily $\dive g=0$; hence the main
result we prove is the following.
\begin{theorem}
  \label{thm:main_theorem}
  Let $\Omega\subset \R^3$ bounded and star-shaped with respect to
  $\overline{B}=\overline{B(0,1)}$.  Let be given $g=(g^{1},g^{2},g^{3})\in
  (C_{D}({\Omega}))^{3}$ such that $\frac{
    \partial g^{i}(x)}{\partial {x_i}} $ exist for $i=1,2,3$, for all $x\in \Omega$, and
  with $\dive g=0$. Then, there exists at least a solution $v=(v^{1},v^{2},v^{3})$ of the
  curl system~\eqref{eq:curl} with homogeneous Dirichlet boundary conditions, such that
  $v\in (C^{1}({\Omega})\cap C(\overline{\Omega}))^{3}$.
\end{theorem}

To conclude the introduction, we finally recall that the boundary value
problem~\eqref{eq:curl} is a system of first order and in addition that the solution is
not unique. Hence, many of the results valid for the Laplacian (or in general
for scalar elliptic equations of the second order) are not directly applicable.

\section{Notation and few basic results}
\label{sec:notation}
We fix now the notation which will be used throughout the paper.  In the sequel we denote
by $B=B(0,1)$ the unit ball in $\R^3$ 
\begin{equation*}
  B(0,1):=\left\{y\in \R^{3}:\ |y|<1\right\},
\end{equation*}
and $\Omega\subset\R^3$ will be an open and bounded, star-shaped with respect to all
points of $\overline{B}$. Whenever we write the representation formulas, we are assuming
the above hypothesis on the open set $\Omega$.

Moreover, let $\psi\in C^\infty_0(\R^3)$ such that $\spt(\psi)\subset B(0,1)$ and
$\int_{\R^3}\psi(y)\,dy=1$.  Together with the customary Lebesgue spaces
$(L^{p}(\Omega),\|\,.\,\|_{L^{p}})$, in this paper we will use the notion of Dini
continuous functions. We recall that a function $f\in C(\overline{\Omega})$ is called Dini
continuous if its modulus of continuity
\begin{equation*}
  \omega(f,\rho):=\sup\{|f(x)-f(y)|\ \text{with }x,y\in\overline{\Omega}\text{ and }
  |x-y|\leq\rho\}, 
\end{equation*}
verifies
\begin{equation*}
  \int_{0}^{\diam(\Omega)}\frac{\omega(f,\rho)}{\rho}\,d\rho<+\infty.
\end{equation*}
The space of Dini continuous functions is denoted by $C_{D}({\Omega})$, is a Banach space
when endowed with the norm
\begin{equation*}
  \|f\|_{C_{D}}:=\max_{x\in \overline{\Omega}}|f(x)|+
  \int_{0}^{\diam(\Omega)}\frac{\omega(f,\rho)}{\rho}\,d\rho,
\end{equation*}
and it is compactly embedded into the space of uniformly continuous functions
$C(\overline{\Omega})$. 

The classical result due to Dini~\cite{Dini1902} (which followed on the wake of earlier
results on the convergence of Fourier series~\cite{Dini1880}) states that if $f\in
C_{D}({\Omega})$, and if $\Omega$ is a smooth domain, then the solution of the Poisson
problem
\begin{equation*}
  \begin{aligned}
    -\Delta u&=f\qquad \text{in }\Omega,
    \\
    u&=0\qquad\text{ on }\partial\Omega,
  \end{aligned}
\end{equation*}
is in $C^{2}(\Omega)$. The extension to elliptic problems as well as to the boundary
regularity (provided that the domain is smooth enough) seems part of the folklore in the
classical theory of elliptic partial differential equations,
see~\cite[Pb.~4.2]{GT1998}. On the other hand, it is well-known that the result of
continuity of second derivatives of $u$ is false if $f$ is just in
$C(\overline{\Omega})$. It is also clear that $C^{0,\alpha}({\Omega})\subseteq
C_{D}({\Omega})$ for all $0<\alpha\leq1$, if $\Omega$ is regular and bounded, hence the
result in the Dini-continuous setting is sharper than those of H\"older and Schauder in
the case of data which are H\"older continuous.
\section{The representation formula}
As explained in the introduction, we will prove the existence of a classical solution by
means of explicit representation formulas \textit{\`a la} Sobolev, as developed by
Bogovski\u\i. For the reader's convenience we recall such formulas in this section and we
take also the occasion to make some remarks on the role of the support of the involved
functions. In particular, as in our forthcoming companion paper on the divergence
equation~\cite{BL2017}, here we use an approach which is slightly different from the ones
previously employed in the literature and which allows also to treat data which cannot be
approximated by $C^{\infty}_{0}(\Omega)$ functions, as it happened in the $L^{p}(\Omega)$
(or even Orlicz spaces) cases treated in the existing literature.

The formulas we will use are a variant of the ``cubature formulas'' developed by
Sobolev~\cite{Sob1938,Sob1974} (see also the review in Burenkov~\cite{Bur1998}), which
have been adapted to the context of the curl operator by Griesinger~\cite{Gri1990b}. We
also observe that the representation formulas from~\cite{Gri1990b} are more general that
the one we use here, since they are valid for all space dimensions and for the curl
operator as well as for its adjoint. (Recall also that in three space dimensions $\curl$
and its adjoint are the same, modulo a change of sign). Anyway, in the specific case of
$\R^3$, the one we are mostly interested to, the representation formula developed
in~\cite[Theorem~3.2]{Gri1990b} is the following
\begin{equation}
  \label{eq:Regina}  
  (\mathcal{R}g)^{k}(x):=-\epsilon_{i j k}\int_{\R^3}z_i
  \,g^{j}(x-z)\int_1^\infty\psi(x-z+t z)\,t(t-1)\,dt dz, 
\end{equation}
for $k=1,2,3$ and where $\epsilon_{i j k}$ is the totally anti-symmetric Ricci tensor such
that the vector product is $(v\times w)^{i}=\epsilon_{i j k}v^{j} w^{k}$, when written in
orthogonal coordinates. We always use the Einstein convention of summation over repeated
indices.  The vector $g$ is intended to be extended by zero outside $\Omega$.

\begin{remark}
  in this paper we consider only the case $n=3$, since it is the most relevant in terms of
  applications to mathematical fluid mechanics. Nevertheless the same approach can be
  easily adapted 
  also to the problem in  $\Omega\subset \R^{n}$, for $n>3$.
\end{remark}

It will be useful to rewrite the representation formula~\eqref{eq:Regina} as follows
\begin{equation}
  \label{eq:rot-1}
  (\mathcal{R}g)^{k}(x)=-\epsilon_{i j
    k}\int_{\Omega}(x-y)_i\,g^{j}(y)\int_1^\infty\psi(y+\alpha(x-y))\,\alpha(\alpha-1) 
  \,d\alpha dy,
\end{equation}
by observing that the above integral involves the values of $g$ only over $\Omega$. We can
also use a more compact form by writing
\begin{equation}
  \label{eq:Regina-2}  
  (\mathcal{R}g)^{k}(x):=-\epsilon_{i j k}\int_{\Omega}\,g^{j}(y) N_i(x,y)\,dy,
\end{equation}
where 
\begin{equation}
  \label{eq:enne}
  N_i(x,y):=(x-y)_i\int_1^\infty\psi(y+\alpha(x-y))\,\alpha(\alpha-1)\,d\alpha.
\end{equation}

For the reader's convenience, we recall also the Bogovski\u\i{} formula defining the
solution of the divergence equation, since the two formulas are strictly connected. Let be
given $F:\Omega\to\R$ such that $\int_{\Omega}F(x)\,dx=0$, then a solution to the
boundary value problem
\begin{equation}
  \label{eq:div}
  \begin{aligned}
    \dive u&=F\qquad\text{in }\Omega,
    \\
    u&=0\qquad\text{on }\partial\Omega,
  \end{aligned}
\end{equation}
which we will denote by 
\begin{equation*}
  u(x)=\mathcal{B}F(x),
\end{equation*}
(where $\mathcal{B}$ denotes as usual the Bogovski\u\i{} operator) can be written as follows
\begin{equation}
  \label{eq:cal-B}
  (\mathcal{B}F)^{i}(x)= u^{i}(x)=\int_{\Omega}F(y)   \widetilde{N}_{i}(x,y)\,dy,
\end{equation}
where 
\begin{equation}
 \label{eq:ennetilde}
 \widetilde{N}_{i}(x,y):=(x-y)_i\int_1^\infty\psi(y+\alpha(x-y))\,\alpha^{2}\,d\alpha.
\end{equation}
From the similarity between the two formulas~\eqref{eq:Regina-2}-\eqref{eq:enne}
and~\eqref{eq:cal-B}-\eqref{eq:ennetilde} it is clear that many results valid for the
operator $\mathcal{R}$ can be deduced by corresponding ones already proved for the
operator $\mathcal{B}$.

We make some remarks and state some results without complete proofs, since some of the
results are well-known~\cite{Gri1990a,Gri1990b} and detailed proof of some new points (at
least for the companion operator $\mathcal{B}$ which is a right inverse of the divergence) can
be found with full details in the forthcoming paper~\cite{BL2017}.

The first results we have is the  following.
\begin{proposition}
  Let be given $g\in (L^q(\Omega))^3$, for some $q>3$. Then, the following formulas are
  all equivalent to~\eqref{eq:Regina}:
  \begin{equation*}
    \begin{aligned}
      (\mathcal{R}g)^{k}(x)&=-\epsilon_{i j
        k}\int_{\Omega}(x-y)_i\,g^{j}(y)\int_1^\infty\psi(y+\alpha(x-y))\,\alpha(\alpha-1)
      \,d\alpha dy,
      \\
      (\mathcal{R}g)^{k}(x)&= -\epsilon_{i j
        k}\int_{\Omega}\frac{(x-y)_i}{|x-y|^3}\,g^{j}(y)\int_{|x-y|}^\infty\psi\big(y+\xi\frac{x-y}{|x-y|}\big)\,
      \xi(\xi-|x-y|)\,d\xi dy,
      \\
      (\mathcal{R}g)^{k}(x)&= -\epsilon_{i j
        k}\int_{\Omega}\frac{(x-y)_i}{|x-y|^3}\,g^{j}(y)\int_{0}^\infty\psi\big(y+r\frac{x-y}{|x-y|}\big)\,r(r+|x-y|)\,dr
      dy,
      \\
      (\mathcal{R}g)^{k}(x)&= -\epsilon_{i j
        k}\int_{x-\Omega}\frac{x_i}{|x|^3}\,g^{j}(x-z)\int_{0}^\infty\psi\big(y+r\frac{z}{|z|}
      \big)\,r(r+|z|) \,dr dz,
      \\
      (\mathcal{R}g)^{k}(x)&=-\epsilon_{i j
        k}\int_{x-\spt g}\frac{x_i}{|x|^3}\,g^{j}(x-z)\int_{0}^{1+\diam(\Omega)}\psi\big(y+r\frac{z}{|z|}\big)\,r(r+|z|)
      \,dr dz.
    \end{aligned}
  \end{equation*}
%
\end{proposition}
\begin{proof}
  The equivalence of the integral formulas is simply proved by applying the formula of a
  change of variables for integrals. The
  fact that the formula is well-defined for $g\in (L^{q}(\Omega))^{3}$ follows from
  Lemma~\ref{lem:first-lemma} below, which implies that for all $x\in \Omega$ it holds $N(x,y)\in
  L^{p}_{loc}(\R^{3},dy)$, for all $p<3$.
\end{proof}
\begin{remark}
  The above result shows that for $g\in (C(\overline{\Omega}))^{3}\subset
  (L^{\infty}(\Omega))^{3}\subset (L^{q}(\Omega))^{3}$, for all $q\geq1$ since $\Omega$ is
  bounded, the representation formula can be applied directly to the vector valued
  function $g$ itself, and not only through smooth approximating sequences.
\end{remark}
The proof of further properties of the operator $\mathcal{R}$ is based on the following
two lemmas.
\begin{lemma}
  \label{lem:first-lemma}
  We can rewrite the functions $N_i(x,y)$ as follows
  \begin{equation*}
    N_{i}(x,y):=\frac{(x-y)_i}{|x-y|^3}\int_{|x-y|}^\infty\psi\big(y+\xi\frac{x-y}{|x-y|}\big)\,\xi(\xi-|x-y|)\,d\xi
    \qquad   \forall\,x\not=y,\ x,y\in\R^{3},
\end{equation*}
and then for $i=1,2,3$
\begin{equation*}
  \exists\, C=C(\diam(\Omega),\|\psi\|_{\infty}):\quad |  N_{i}(x,y)|\leq\frac{C}{|x-y|^{2}}\qquad
  \forall\,x\not=y,\ x,y\in\R^{3}.
\end{equation*}
\end{lemma}
\begin{proof}
  The proof is based on the observation that the function $\psi$ is bounded and
  $\psi\big(y+\xi\frac{x-y}{|x-y|}\big)$ is zero for all $\xi\in \R^+$ such that
  $\xi>1+\diam\,\Omega$, and therefore the integral of a continuous function over a
  compact set is bounded.
\end{proof}
A very basic fact, which has not been highlighted in the literature is that the following
result is valid.
\begin{lemma}
\label{ref:lemma-second}
The functions $N_{i}(x,y)$ for $i=1,2,3$ are such that
  \begin{equation*}
    N_{i}(x,y)\equiv 0\qquad\forall \,y\in \Omega\quad \text{and}\quad\forall\,x\in \R^{3}\backslash \Omega.
  \end{equation*}
\end{lemma}
\begin{proof}
  It is easy to check that if $x\notin\Omega$ and
  $\psi\left(y+\xi\frac{x-y}{|x-y|}\right)\neq 0$ holds true for some $\xi>|x-y|$, then
  $y\notin\Omega$. From this fact the thesis follows directly.
\end{proof}
In fact, Lemma~\ref{ref:lemma-second} implies that the homogeneous Dirichlet boundary
conditions are satisfied for any function for which~\eqref{eq:Regina} makes sense.

Immediate corollaries of the above two lemmas, are the following propositions
\begin{proposition}
  If $g\in (L^{q}(\Omega))^{3}$, for some $q>3$, then $\mathcal{R}g(x)$ is well defined for all
  $x\in \R^{3}$ and moreover
  \begin{equation*}
    (\mathcal{R}g)(x)\equiv0\qquad\forall x\in \R^{3}\backslash \Omega.
  \end{equation*}
  In particular, we have that $\mathcal{R}g(x)=0$ for all $x\in \partial\Omega$.
\end{proposition}
\begin{proposition}
  If $g\in (L^{q}(\Omega))^{3}$, for some $q>3$, then 
  \begin{equation*}
    |\mathcal{R}g(x)|\leq c\|g\|_{L^{q}(\Omega)}\qquad \forall\,x\in\R^{3},
\end{equation*}
where $c$ depends only on $\psi$, $\diam(\Omega)$, and $q$. 
\end{proposition}
\begin{proposition}
  If $g\in (C^{\infty}_{0}(\Omega))^{3}$, then $\mathcal{R}g\in  (C^{\infty}_{0}(\Omega))^{3}$.
\end{proposition}
This latter result is readily obtained by differentiating under the integral sign, see
also Griesinger~\cite{Gri1990a,Gri1990b} and Borchers and Sohr~\cite{BS1990}, since under
the above hypotheses all the calculations are completely justified.

Moreover we have also the following result
\begin{proposition}
\label{prop:proposition4}
  Let $g\in (L^{q}(\Omega))^{3}$, for some $q>3$. Then, $\mathcal{R}g\in
  (C(\overline{\Omega}))^{3}$, with $\mathcal{R}g_{|\partial\Omega}=0$.
\end{proposition}

\begin{proof}
  The proof is obtained by approximating in $L^{q}(\Omega)$ the vector $g$ with a sequence
  $\{g_{m}\}_{m\in\N}\subset (C^{\infty}_{0}(\Omega))^{3}$ and by observing that the
  resulting sequence of vectors $\{\mathcal{R}g_{m}\}_{m\in\N}$ belongs to $\subset
  (C^{\infty}_{0}(\Omega))^{3}$ and it converges uniformly, as $m\to+\infty$, to the
  vector field $\mathcal{R}g$, in the whole space $\R^{3}$.
\end{proof}
\subsection{On the validity of the representation}
Once we have shown that $\mathcal{R}g$ vanishes at the boundary, it is relevant now to
show that the operator $\mathcal{R}$ represents a right inverse of the $\curl$ operator. The
proof we give here is different from that provided in~\cite{Gri1990b,Gri1990a}, where the
$L^{p}$-approach is based on a non-continuous truncation of the kernel and on the analysis
of the surface integral deriving from integration by parts.

In this section the open set $\Omega$ is star-shaped with respect to
$\overline{B(0,1)}$. We show now that the formula of representation provides a solution to
the curl equation.  This follows by using a very classical tool (introduced by
Korn~\cite{KR1}) of truncating in a smooth way the singularity of the kernel. Then, we
show that we can work on smooth functions, and only at the end we pass to the limit,
proving properties which can be derived by the uniform convergence of the approximating
sequence.

To this end, let be given a monotone non-decreasing $\eta\in C^1(\R^+)$ such that
\begin{equation*}
  \eta(s)=\left\{
    \begin{aligned}
      &      0\qquad \text{for }s\in[0,1],
      \\
      &1\qquad \text{for }s\in[2,+\infty[,
    \end{aligned}
\right.
\end{equation*}
and $|\eta'(s)|\leq2$ for all $s\in\R^+$. 

Then, let us fix $x\in \Omega$ and let be given
$0<\epsilon<\textrm{dist}(x,\partial\Omega)$. We start considering the representation
formula for $g\in (C^\infty_0(\Omega))^3$ and we define the operator $\mathcal{R}^{\epsilon}$
(which is the ``$\epsilon$-regularized'' version of the operator $\mathcal{R}$), obtained
by truncating in a smooth way the kernel $N_i(x,y)$ near its singularity at $x=y$:
\begin{equation*}
  (\mathcal{R}^{\epsilon}
  g)^{k}(x):=-\epsilon_{i j k}\int_{\Omega}(x-y)_i\,g^{j}(y)\int_1^\infty\psi(y+\alpha(x-y))\,\alpha(\alpha-1)\,  
  \eta\Big(\frac{|x-y|}{\epsilon}\Big)\,d\alpha dy.
\end{equation*}
This approximation allows us to freely perform all needed manipulations, since all terms
in the above integral are smooth and bounded over $\Omega$.  Before performing the various
derivatives we observe that the following crucial identity is valid.
  \begin{lemma}
    \label{lem:change-derivatives}
    By evaluating the partial derivatives we get, for all $j=1,2,3$ 
    \begin{equation}
      \label{eq:integrate_by_parts}
      \begin{aligned}
        &\frac{\partial}{\partial
          x_j}\psi(y+\alpha(x-y))=\partial_j\psi(y+\alpha(x-y))\,\alpha,
        \\
        &\text{and}
        \\
        &\frac{\partial}{\partial
          y_j}\psi(y+\alpha(x-y))=\partial_j\psi(y+\alpha(x-y))\,(1-\alpha),
      \end{aligned}
    \end{equation}
    where, to simplify the notation, here and in the sequel we use the symbol $\partial_j
    \psi$ to denote the derivative with respect to the $j$-th argument of the function
    $\psi$ that is
    \begin{equation*}
      \partial_j\psi(y+\alpha(x-y)):=\partial_j\psi_{\,|y+\alpha(x-y)}.
    \end{equation*}
%
%
  Then, for any fixed $x,y\in \Omega,\;x\neq y$ and for all $i,j=1,2,3$ it follows 
  \begin{equation*}
    \partial_{x_{j}}N_{i}(x,y)\;=\;
    (x_{i}-y_{i})\int_{1}^{\infty}\partial_{j}\psi(y+\alpha(x-y))\,\alpha(\alpha-1)\,d\alpha
    \;-\;\partial_{y_{j}}N_{i}(x,y).
  \end{equation*}
\end{lemma}
\begin{proof}
  Since $\left|y+\alpha(x-y)\right|\geq 1$ for $\alpha\geq(1+|y|)/|x-y|$ the integrand is
  bounded on a compact subset of $\R$ and we can differentiate under the sign of
  integral. By using Lemma~\ref{lem:change-derivatives} we obtain by direct calculations
  the proof.
\end{proof}
We now show  the following result.
\begin{proposition}
\label{prop:take-derivatives}
Let  be given a scalar $\psi\in C^{\infty}_{0}(\R^{3})$ such that $\int_{\R^{3}}\psi(y)\,dy=1$ and
$\spt\psi\subset B(0,1)$. Let be given $g\in (C(\overline{\Omega}))^{3}$ with $\dive g\in
L^{q}(\Omega)$, for some $q>3$. Then,
\begin{equation*}
  \lim_{\epsilon\to0}\curl (\mathcal{R}^{\epsilon} g)(x)=g(x)+\mathcal{B}[\dive g]\qquad  \forall\,x\in \Omega,
\end{equation*}
where $\mathcal{B}$ is the Bogovski\u\i{} operator, which gives a solution of the
divergence equation with zero boundary conditions and which is defined by
formula~\eqref{eq:cal-B}.
\end{proposition}
We prove Proposition~\ref{prop:take-derivatives} under an additional hypothesis, which is
mostly interesting for our purposes.
\begin{theorem}
  \label{thm:1}
  Let be given a scalar $\psi\in C^{\infty}_{0}(\R^{3})$ such that $\int_{\R^{3}}\psi(y)\,dy=1$ and
  $\spt\psi\subset B(0,1)$. Let be given $g\in (C(\overline{\Omega}))^3$ such that $\frac{
    \partial  g^{i}(x)}{\partial {x_i}} $ exist for $i=1,2,3$ and $\dive g=0$. Then,
  \begin{equation*}
    \lim_{\epsilon\to0}\curl (\mathcal{R}^{\epsilon} g)(x)=g(x)\qquad \forall\,x\in \Omega.
  \end{equation*}
\end{theorem}

%
\begin{proof}[Proof of Theorem~\ref{thm:1}]
  The proof is obtained by explicitly taking the curl of the representation formula for
  $\mathcal{R}g$ in~\eqref{eq:rot-1} in its regularized form $\mathcal{R}^{\epsilon}g$. We
  obtain then
  \begin{equation*}
    \begin{aligned}
      & [\curl (\mathcal{R}^{\epsilon} g)]^{i}(x)
      \\
      &\  =-\epsilon_{i j k}\epsilon_{k l m}\frac{\partial}{\partial
        x_j}\int_{\Omega}(x-y)_l\,g^{m}(y)\int_1^\infty\psi(y+\alpha(x-y))\,\alpha(\alpha-1)
      \,\eta\Big(\frac{|x-y|}{\epsilon}\Big)\,d\alpha dy
      \\
      &\ =(\delta_{i m}\delta_{j l}-\delta_{i l}\delta_{j m})\times
      \\
      &\ \ \times\Big[\int_{\Omega}(x-y)_l\,g^{m}(y)\int_1^\infty\partial_{j}\psi(y+\alpha(x-y))\,\alpha^2(\alpha-1)
      \,\eta\Big(\frac{|x-y|}{\epsilon}\Big)\,d\alpha dy
      \\
      &\ \ +\int_{\Omega}\delta_{j l}\,g^{m}(y)\int_1^\infty\psi(y+\alpha(x-y))\,\alpha(\alpha-1)
      \,\eta\Big(\frac{|x-y|}{\epsilon}\Big)\,d\alpha dy
      \\
      &\ \ +\int_{\Omega}(x-y)_l\,g^{m}(y)\int_1^\infty\psi(y+\alpha(x-y))\,\alpha(\alpha-1)
      \,\eta'\Big(\frac{|x-y|}{\epsilon}\Big)\frac{(x-y)_j}{\epsilon|x-y|}\,d\alpha dy \Big],
      \\
      & =:A+B+C.
    \end{aligned}
  \end{equation*}
  Let us first consider  the term $A$, which can be written more explicitly as follows:
\begin{equation*}
  \begin{aligned}
    A&=\int_{\Omega}(x-y)_j\,g^{i}(y)\int_1^\infty{\partial_{j}}\psi(y+\alpha(x-y))\,\alpha^2(\alpha-1)
    \,\eta\Big(\frac{|x-y|}{\epsilon}\Big)\,d\alpha dy
    \\
    &\quad-\int_{\Omega}(x-y)_i\,g^{j}(y)\int_1^\infty{\partial_j}\psi(y+\alpha(x-y))\,\alpha^2(\alpha-1)
    \,\eta\Big(\frac{|x-y|}{\epsilon}\Big)\,d\alpha dy,
    \\
    &=:A_1+A_2.
  \end{aligned}
\end{equation*}
Observe now that
\begin{equation*}
  A_1=\int_{\Omega}\,g^{i}(y)\int_1^\infty\Big[\frac{d}{d\alpha}\psi(y+\alpha(x-y))\Big]\,\alpha^2(\alpha-1)
  \,\eta\Big(\frac{|x-y|}{\epsilon}\Big)\,d\alpha dy,
\end{equation*}
hence, integrating by parts with respect to $\alpha$, we obtain
\begin{equation*}
  \begin{aligned}
    A_1&=\int_{\Omega}\,g^{i}(y)\psi(y+\alpha(x-y))\,\alpha^2(\alpha-1)\Big|_{\alpha=1}^{\alpha=+\infty}
    \eta\Big(\frac{|x-y|}{\epsilon}\Big)\,d\alpha dy
    \\
    &-\int_{\Omega}\,g^{i}(y)\int_1^\infty\psi(y+\alpha(x-y))\,\Big[\frac{d}{d\alpha}\alpha^2(\alpha-1)\Big]
    \,\eta\Big(\frac{|x-y|}{\epsilon}\Big)\,d\alpha dy,
  \end{aligned}
\end{equation*}
and the boundary term vanishes identically. 

Concerning $A_2$ we observe that we can use~\eqref{eq:integrate_by_parts} to interchange
the derivative with respect to the $x$ variables into one in the $y$ variables, to write
(after integration by parts and again dropping the boundary terms which vanish)
\begin{equation*}
  \begin{aligned}
    A_2&=\int_{\Omega}(x-y)_i\,g^{j}(y)\int_1^\infty\frac{\partial}{\partial
      y_j}\psi(y+\alpha(x-y))\,\alpha^2 \,\eta\Big(\frac{|x-y|}{\epsilon}\Big)\,d\alpha dy
    \\
    &=-\int_{\Omega}(x-y)_i\,\frac{\partial g^{j}(y)}{\partial
      y_j}\int_1^\infty\psi(y+\alpha(x-y))\,\alpha^2
    \,\eta\Big(\frac{|x-y|}{\epsilon}\Big)\,d\alpha dy
    \\
    &\qquad+\int_{\Omega}\delta_{i j}\,g^{j}(y)\int_1^\infty\psi(y+\alpha(x-y))\,\alpha^2
    \,\eta\Big(\frac{|x-y|}{\epsilon}\Big)\,d\alpha dy
    \\
    &\qquad +\int_{\Omega}(x-y)_i\,g^{j}(y)\int_1^\infty\psi(y+\alpha(x-y))\,\alpha^2
    \,\eta'\Big(\frac{|x-y|}{\epsilon}\Big)\frac{(x-y)_j}{\epsilon|x-y|}\,d\alpha dy.
  \end{aligned}
\end{equation*}
Hence, we obtain that  
\begin{equation*}
  \begin{aligned}
    A_2&=-\int_{\Omega}(x-y)_i\,\big(\dive
    g(y)\big)\int_1^\infty\psi(y+\alpha(x-y))\,\alpha^2
    \,\eta\Big(\frac{|x-y|}{\epsilon}\Big)\,d\alpha dy
    \\
    &\qquad+\int_{\Omega}\,g^{i}(y)\int_1^\infty\psi(y+\alpha(x-y))\,\alpha^2
    \,\eta\Big(\frac{|x-y|}{\epsilon}\Big)\,d\alpha dy
    \\
    &\qquad
    +\int_{\Omega}\,g^{j}(y)\frac{(x-y)_i(x-y)_j}{\epsilon|x-y|}\int_1^\infty\psi(y+\alpha(x-y))\,\alpha^2
    \,\eta'\Big(\frac{|x-y|}{\epsilon}\Big)\,d\alpha \frac{dy}{\epsilon}.
  \end{aligned}
\end{equation*}
Then, adding together the two formulas, we finally proved that
\begin{equation*}
  \begin{aligned}
    A&=A_1+A_2
    \\
    &= -2\int_{\Omega}\,g^{i}(y)\int_1^\infty\psi(y+\alpha(x-y))\,\alpha(\alpha-1)
    \,\eta\Big(\frac{|x-y|}{\epsilon}\Big)\,d\alpha dy
    \\
    &\qquad+\int_{\Omega}(x-y)_i\,\dive g(y)\int_1^\infty\psi(y+\alpha(x-y))\,\alpha^2
    \,\eta\Big(\frac{|x-y|}{\epsilon}\Big)\,d\alpha dy
    \\
    &\qquad
    +\int_{\Omega}\,g^{j}(y)\frac{(x-y)_i(x-y)_j}{\epsilon|x-y|}\int_1^\infty\psi(y+\alpha(x-y))\,\alpha^2
    \,\eta'\Big(\frac{|x-y|}{\epsilon}\Big)\,d\alpha \frac{dy}{\epsilon}.
  \end{aligned}
\end{equation*}
Let us now consider the term $B$. We have
\begin{equation*}
  B=(\delta_{i m}\delta_{j l}-\delta_{i l}\delta_{j m})\int_{\Omega}\delta_{j l}\,g^{m}(y)
  \int_1^\infty\psi(y+\alpha(x-y))\,\alpha(\alpha-1)\,\eta\Big(\frac{|x-y|}{\epsilon}\Big)\,d\alpha
  dy, 
\end{equation*}
and we observe that since $\delta_{j l}\delta_{j l}=3$, it follows that 
\begin{equation*}
  \begin{aligned}
    B=&\,3\int_{\Omega}\,g^{i}(y)\int_1^\infty\psi(y+\alpha(x-y))\,\alpha(\alpha-1)
    \,\eta\Big(\frac{|x-y|}{\epsilon}\Big)\,d\alpha dy
    \\
    &\qquad-\int_{\Omega}\delta_{i l}\,g^{l}(y)\int_1^\infty\psi(y+\alpha(x-y))\,\alpha(\alpha-1)
    \,\eta\Big(\frac{|x-y|}{\epsilon}\Big)\,d\alpha dy
    \\
    =&\,2\int_{\Omega}\,g^{i}(y)\int_1^\infty\psi(y+\alpha(x-y))\,\alpha(\alpha-1)
    \,\eta\Big(\frac{|x-y|}{\epsilon}\Big)\,d\alpha dy.
  \end{aligned}
\end{equation*}
Concerning the term $C$ we can rewrite it as follows
  \begin{equation*}
    \begin{aligned}
      C&=(\delta_{i m}\delta_{j l}-\delta_{i l}\delta_{j m})\times
      \\
      &\qquad \times
      \int_{\Omega}(x-y)_l\,g^{m}(y)\frac{(x-y)_j}{\epsilon|x-y|}\int_1^\infty\psi(y+\alpha(x-y))\,\alpha(\alpha-1) 
      \,\eta'\Big(\frac{|x-y|}{\epsilon}\Big)\,d\alpha dy 
      \\
      &=\int_{\Omega}\frac{(x-y)_j(x-y)_j}{\epsilon|x-y|}\,g^{i}(y)\int_1^\infty\psi(y+\alpha(x-y))\,\alpha(\alpha-1)
      \,\eta'\Big(\frac{|x-y|}{\epsilon}\Big)\,d\alpha dy 
      \\
      &\qquad-\int_{\Omega}\frac{(x-y)_i(x-y)_j}{\epsilon|x-y|}\,g^{j}(y)\int_1^\infty\psi(y+\alpha(x-y))\,\alpha(\alpha-1)
      \,\eta'\Big(\frac{|x-y|}{\epsilon}\Big)\,d\alpha dy 
      \\
      &=\int_{\Omega}\frac{|x-y|}{\epsilon}\,g^{i}(y)\int_1^\infty\psi(y+\alpha(x-y))\,\alpha(\alpha-1)
      \,\eta'\Big(\frac{|x-y|}{\epsilon}\Big)\,d\alpha dy 
      \\
      &\qquad-\int_{\Omega}\frac{(x-y)_i(x-y)_j}{\epsilon|x-y|}\,g^{j}(y)\int_1^\infty\psi(y+\alpha(x-y))\,\alpha(\alpha-1)
      \,\eta'\Big(\frac{|x-y|}{\epsilon}\Big)\,d\alpha dy.
    \end{aligned}
\end{equation*}
Then, adding together $A+B+C$ from the resulting formulas, and by using that $\dive g=0$,
we  finally obtain
\begin{equation*}
  \begin{aligned}
    & [\curl (\mathcal{R}^{\epsilon} g)]^i(x)
    \\
    &=\int_{\Omega}\frac{|x-y|}{\epsilon}\,g^{i}(y)\,\eta'\Big(\frac{|x-y|}{\epsilon}\Big)\int_1^\infty\psi(y+\alpha(x-y))\,\alpha^2
    \,d\alpha dy
    \\
    &\qquad-
    \int_{\Omega}\frac{|x-y|}{\epsilon}\,g^{i}(y)\,\eta'\Big(\frac{|x-y|}{\epsilon}\Big)\int_1^\infty\psi(y+\alpha(x-y))\,\alpha
    \,d\alpha dy
    \\
    &\qquad+\int_{\Omega}\frac{(x-y)_i(x-y)_j}{\epsilon|x-y|}\,g^{j}(y)
    \,\eta'\Big(\frac{|x-y|}{\epsilon}\Big)\int_1^\infty\psi(y+\alpha(x-y))\,\alpha
    \,d\alpha dy,
    \\
    &=:S^\epsilon_{1}(x)+S^\epsilon_2(x)+S^\epsilon_3(x).
  \end{aligned}
\end{equation*}
We now take the limit as $\epsilon\to0^+$ and we start with the term $S^\epsilon_1(x)$, for which we
make the change of variables $\alpha=\frac{\xi}{|x-y|}$ to obtain 
\begin{equation*}
  S^\epsilon_1(x)=\int_{\Omega}g^{i}(y)\,\eta'\Big(\frac{|x-y|}{\epsilon}\Big)\int_{|x-y|}^\infty
  \psi\Big(y+\xi\frac{x-y}{|x-y|}\Big)\,\frac{\xi^2}{|x-y|^2} 
  \,d\xi \frac{dy}{\epsilon}.
\end{equation*}
Then, with the further change of variables $r=\xi-|x-y|$ we obtain
\begin{equation*}
  S^\epsilon_1(x)=
\int_{1<|x-y|<2}g^{i}(y)\,\eta'\Big(\frac{|x-y|}{\epsilon}\Big)\int_0^\infty\psi\Big(x+r\frac{x-y}{|x-y|}\Big)\,\frac{(r+|x-y|)^2}{|x-y|^2} 
  \,dr  \frac{dy}{\epsilon}.
\end{equation*}
If we set $z=\frac{x-y}{\epsilon}$, then 
\begin{equation*}
  \begin{aligned}
    S^\epsilon_1(x)&=\int_{1<|z|<2}g^{i}(x-\epsilon
    z)\,\eta'(z)\int_0^\infty\psi\Big(x+r\frac{z}{|z|}\Big)\,\frac{(r+\epsilon^2|z|)^2}{\epsilon^2|z|^2}
    \,dr \frac{\epsilon^3 dz}{\epsilon}
    \\
    &=\int_{1<|z|<2}g^{i}(x-\epsilon
    z)\,\eta'(z)\int_0^\infty\psi\Big(x+r\frac{z}{|z|}\Big)\,\frac{(r+\epsilon|z|)^2}{|z|^2}
    \,dr dz,
  \end{aligned}
\end{equation*}
and, by using the continuity of $g$, we get (by the dominated convergence theorem) that
the latter integral converges, as $\epsilon\to0$, in such a way that
\begin{equation*}
  \lim_{\epsilon\to0} S^\epsilon_1(x)=g^{i}(x)\int_{1<|z|<2}
\frac{\,\eta'(z)}{|z|^2}\int_0^\infty\psi\Big(x+r\frac{z}{|z|}\Big)\,r^{2}
    \,dr dz.
\end{equation*}
Finally, by introducing the radial $ \rho=|z|$ and angular $u=z/|z|$ coordinates
respectively  one gets
\begin{equation*}
  \begin{aligned}
    &\int_{1<|z|<2}
    \frac{\,\eta'(z)}{|z|^2}\int_0^\infty\psi\Big(x+r\frac{z}{|z|}\Big)\,r^{2} \,dr
    dz
    \\
    &=\int_{1}^{2}\eta'(\rho)\,d\rho\int_{S^{2}}\,du\int_{0}^{\infty}\psi(x+r u)\,r^{2}\,dr
    =\left(\eta(2)-\eta(1)\right) \int_{\R^{3}}\psi(w)\,dw=1.
  \end{aligned}
\end{equation*}
Therefore,
\begin{equation*}
  \lim_{\epsilon\to0} S^\epsilon_1(x)=\left(\eta(2)-\eta(1)\right)  g^{i}(x)
  \int_{\R^{3}}\psi(w)\,dw=g^{i}(x). 
\end{equation*}
With the same changes of variables we also get 
\begin{equation*}
  \begin{aligned}
    S^\epsilon_2(x)&=-\epsilon\int_{1<|z|<2}g^{i}(x-\epsilon z)\frac{\eta'(z)}{
      |z|}\int_0^\infty\psi\Big(x+r\frac{z}{|z|}\Big)\,(r+\epsilon|z|) \,dr
    dz=\mathcal{O}(\epsilon),
   \\ 
   S^\epsilon_3(x)&=\epsilon\int_{1<|z|<2}g^{i}(x-\epsilon
    z)\frac{\eta'(z) z_i z_j}{|z|^3}\int_0^\infty\psi\Big(x+r\frac{z}{|z|}\Big)\,(r+\epsilon|z|)
    \,dr dz=\mathcal{O}(\epsilon),
  \end{aligned}
\end{equation*}
ending the proof.
\end{proof}
\section{On the regularity of the solutions}
In this section we prove the main result of this paper, namely the existence of a
classical solution to the curl system~\eqref{eq:curl}, under the assumption that $g\in
(C_D(\Omega))^3$.

We first recall the following result, which is proved in
Griesinger~\cite[Thm.~3.5]{Gri1990b}, and which follows directly by inspection of the
kernel $N(x,y)$
\begin{lemma}
  There exist functions $K_{i j}$ and $G_{i j}$ for any $i,j=1,2,3$ such that
\begin{equation*}
\partial_{x_{j}}N_{i}(x,y)\;=K_{i j}(x,x-y)+G_{i j}(x,y),
\end{equation*}
where $K_{i j}(x,\cdot)$ is a Calder\`{o}n-Zygmund singular kernel and $G_{i j}$ is a weakly
singular kernel in the sense that, if one sets
\begin{equation*}
k_{i j}(x,z)\equiv |z|^{3}K_{i j}(x,z),
\end{equation*}
there exist constants $c=c(\psi)$ and $M=M(\psi,\,\diam\,\Omega)$ such that:
\begin{enumerate}
\item $ \quad k_{i j}(x,t z)\,=\,k_{i j}(x,z)\qquad\forall\, x\in\Omega,\ \forall\, z\neq
  0,\ \forall\, t>0$;
\item $ \quad \Vert k_{i j}(x,z)\Vert_{L^{\infty}(\Omega\times S^{2})} $  is finite;
\item $ \quad \int_{|z|=1}k_{i j}(x,z)\,dz = 0\quad \forall\, x\in\Omega$;
\item $\quad |G_{i j}(x,y) | \leq c\,(\diam\,\Omega)^{2}|x-y|^{-2}$;
\item $\quad |\partial_{x_{j}}N_{i}(x,y)| \leq M |x-y|^{-3}\quad\forall\, x\in\Omega\quad
  \forall\, y\in\R^{3}\backslash\{x\}$.
\end{enumerate}
\end{lemma}
The whole $L^{p}$-theory for the operator $\mathcal{R}$ follows then by using the above
estimates within the framework of singular integrals. We will not use that theory, but
nevertheless we need to use the above estimates to control the growth of the kernel
$N_{i}(x,y)$, as $x$ gets close to $y$.

As usual in classical potential theory, writing an explicit representation formula for the
first order derivatives of the vector field $\mathcal{R}g$ is a main technical fact. We
will obtain it through a limit of the derivatives of $\mathcal{R}^{\epsilon}g$. To this
end we start by differentiating $\mathcal{R}^{\epsilon}g$, which is smooth. We first
extend $g$ by zero in $\R^{3}\backslash{\Omega}$ We thus obtain 
\begin{equation*}
  \begin{aligned}
    (\mathcal{R}g)^{k}(x)&=-\epsilon_{i j k}\int_{\Omega}\,g^{j}(y) N_i(x,y)
    \,\eta\Big(\frac{|x-y|}{\epsilon}\Big)\,dy
    \\
    &=-\epsilon_{i j k}\int_{B_{R}}\,g^{j}(y) N_i(x,y)
    \,\eta\Big(\frac{|x-y|}{\epsilon}\Big)\,dy, 
  \end{aligned}
\end{equation*}
where $B_{R}$ is a ball of radius $R>0$ large enough such that $\Omega \subset\subset
B_{R}$.  By the previous results (especially
Lemma~\ref{lem:change-derivatives}) 
it follows that
\begin{equation*}
  \begin{aligned}
    \partial_{x_{m}}(\mathcal{R}^{\epsilon}&g)^{k}(x)=-\epsilon_{i j
      k}\int_{B_{R}}\,g^{j}(y)\partial_{x_{m}}\left[N_{i}(x,y)\,\eta\left(\frac{|x-y|}{\epsilon}\right)\right]\,dy
    \\
    &=-\epsilon_{i j k}\int_{B_{R}}\,
    \left[g^{j}(y)-g^{j}(x)\right]\partial_{x_{m}}\left[N_{i}(x,y)\,\eta\left(\frac{|x-y|}{\epsilon}\right)\right]\,dy
    \\
    &\qquad-\epsilon_{i j
      k}g^{j}(x)\int_{B_{R}}\partial_{x_{m}}\left[N_{i}(x,y)\,\eta\left(\frac{|x-y|}{\epsilon}\right)\right]\,dy
    \\
    &=\epsilon_{i j k}\int_{B_{R}}\,
    \left[g^{j}(y)-g^{j}(x)\right]\partial_{x_{m}}\left[N_{i}(x,y)\,\eta\left(\frac{|x-y|}{\epsilon}\right)\right]\,dy
    \\
    &\quad-\epsilon_{i j
      k}g^{j}(x)\int_{B_{R}}\eta\left(\frac{|x-y|}{\epsilon}\right)(x_{i}-y_{i})
    \int_{1}^{\infty}\partial_{m}\psi(y+\alpha(x-y)) \alpha(\alpha-1)\,d\alpha\,dy 
    \\
    &\quad +\epsilon_{i j
      k}g^{j}(x)\int_{B_{R}}\partial_{y_{m}}\left[N_{i}(x,y)\,\eta\left(\frac{|x-y|}{\epsilon}\right)\right]\,dy.
\end{aligned}
\end{equation*}
Since when $\epsilon<\textrm{dist}(\partial
B_{R},\overline{\Omega})$, then $\eta\left(\frac{|x-y|}{\epsilon}\right)=1$, by the
Gauss-Green formula the last integral in the above formula is equal to
\begin{equation*}
  \epsilon_{i j
      k}g^{j}(x)\int_{\partial B_{R}}N_{i}(x,y)\nu_{j}(y)\,d\sigma_{y}.
\end{equation*}

The previous computation suggests to put forward a conjecture about the limit as
$\epsilon$ goes to zero, which will be proved in the next theorem, that is the main result
of the paper. From the theorem below, by recalling Proposition~\ref{prop:proposition4},
it will follow directly Theorem~\ref{thm:main_theorem}.
\begin{theorem} 
  Let $\Omega\subset \R^3$ be a bounded open set, star-shaped with respect to the closed unit
  ball $\overline{B}$ centered at the origin and let $R>0$ large enough to get 
  $\Omega\subset \subset B(0,R)$.  Furthermore, let $\psi\in C^\infty_0(\R^3)$ be such
  that $\int_{\R^{3}}\psi(x)\,dx=1$, and $\spt \psi\subset B(0,1)$. Let be given $g\in
  (C_{D}^{0}(\overline{\Omega}))^3$ such that $\frac{
    \partial g^{i}(x)}{\partial {x_i}} $ exist for $i=1,2,3$, for all $x\in \Omega$,
  and $\dive g=0$.
  
  Let the functions $v_{m}^{k}(x)$ for $k,m=1,2,3$ be defined as follows:
  \begin{equation*}
    \begin{aligned}
      v_{m}^{k}(x):=-\epsilon_{i j k}&\Bigg[\int_{B_{R}}
      \left[g^{j}(y)-g^{j}(x)\right]\partial_{x_{m}}N_{i}(x,y)\,dy
      \\
      &\qquad+g^{j}(x)\int_{B_{R}}(x_{i}-y_{i})\int_{1}^{\infty}\partial_{m}\psi(y+\alpha(x-y))\,\alpha(\alpha-1)\,d\alpha\,dy 
      \\
      &\qquad -g ^{j} (x)\int_{\partial B_{R}}N_{i}(x,y)\nu_{m}(y)d\sigma_{y}\Bigg].
    \end{aligned}
  \end{equation*}
  Then:
  \begin{enumerate}
  \item[1)]\quad $v_{m}^{k}(x)$ is well-defined for all $x\in\Omega$;
    \vspace{2mm}
  \item[2)]\quad $\frac{\partial(\mathcal{R}^{\epsilon}g)^{k}(x)}{\partial {x_{m}}} $
    converges uniformly to $v_{m}^{k}(x) $ on any compact $\Omega'\subset\subset\Omega$;
    \vspace{2mm}
  \item[3)] \quad   $\frac{\partial(\mathcal{R} g)^{k}(x)}{\partial {x_{m}}}\equiv v_{m}^{k}(x)$
    in $\Omega$;
    \vspace{2mm}
  \item[4)] \quad  $\mathcal{R}g\in (C^{1}(\Omega))^3$.
  \end{enumerate}
\end{theorem}
\begin{proof}
  The proof of the above result are based on the accurate analysis of the various
  integrals, for fixed $k,m=1,2,3$. In particular, many of the calculations are very close
  to those used in the analysis of the divergence equation in~\cite{BL2017}, and hence we
  sketch the main points, without giving full details.

To prove 1), fix any $x\in\Omega$. Remark that, after its extension by zero outside
$\Omega$, $g$ still belongs to $(L^{\infty}(\R^{3}))^{3}$. Then, for any
$\epsilon<\textrm{dist}(x,\partial\,\Omega)$, one has
\begin{equation*}
  \begin{aligned}
    \int_{B_{R}} \left|g^{j}(y)-g^{j}(x)\right|\,&\left|\partial_{x_{m}}N_{i}(x,y)\right|\,dy
    \\
    &=\int_{B(x,\epsilon)}\left|g^{j}(y)-g^{j}(x)\right|\left|\partial_{x_{m}}N_{i}(x,y)\right|\,dy
    \\
    &\qquad+\int_{\{|x-y|\geq
      \epsilon\}\cap
      B_{R}}\left|g^{j}(y)-g^{j}(x)\right|\left|\partial_{x_{m}}N_{i}(x,y)\right|\,dy,
\\
&=: D+E.
  \end{aligned}
\end{equation*}
Since $B(x,\epsilon)\subset\Omega$, it follows that, for all $y\in \Omega$ with $y\neq x$, 
\begin{equation*}
  \begin{aligned}
    D&\leq\int_{B(x,\epsilon)}
    \frac{\left|g^{j}(y)-g^j(x)\right|}{|y-x|}|y-x|\;|\partial_{x_{m}}N_{i}(x,y)|\,dy
     \\
     &
    \leq\int_{B(x,\epsilon)} \frac{\omega(g,|y-x|)}{|y-x|}\;\frac{M}{|y-x|^{n-1}}\,dy,
\end{aligned}
\end{equation*}
where $\omega(g,\rho)$ is the modulus of continuity of $g$ in $\Omega$. By introducing the radial
and angular coordinates, we obtain 
\begin{equation*}
D\leq 4\pi M\; \int_{0}^{\epsilon} \frac{\omega(g,\rho)}{\rho}\,d\rho,
\end{equation*}
and, by the hypothesis that $g$ is Dini continuous, the integral is finite.

Furthermore, since both $g$ and $\partial_{x_{m}}N_{i}(x,y)$ are bounded on $\{|x-y|\geq
\epsilon\}$, the term  $E$ is finite as well.

Finally, since $\partial_{m}\psi\in C^{\infty}_{0}(\R^{n})$ and $\spt \partial_{m}\psi
\subset B(0,1)$, it follows that
\begin{equation}
  \label{eq:g-costant}
  \int_{\Omega}(x_{i}-y_{i})\int_{1}^{\infty}\partial_{m}\psi(y+\alpha(x-y))\,\alpha(1-\alpha)\,d\alpha\,dy, 
\end{equation}
is the value of the same representation formula~\eqref{eq:Regina} corresponding to the
bounded function $g^{j}\equiv 1$, evaluated by using $\partial_{m}\psi\in
C^{\infty}_{0}(\R^{3)}$ instead of $\psi$. We observe that since
$\spt \partial_{m}\psi\subset B$, the properties proved in
Proposition~\ref{prop:proposition4} are still valid. Hence, by the previous results, also
the integral \eqref{eq:g-costant} is globally bounded, and statement 1) follows.

To prove 2), fix any $\Omega'\subset\subset\Omega$. Thus, for all $x\in\Omega'$ and
$\epsilon>0$ such that $2\epsilon<\textrm{dist}(\overline{\Omega'},\partial \Omega)$, it
follows that
\begin{equation*}
  \begin{aligned}
    & |\partial_{x_{m}}(\mathcal{R}^{\epsilon}g)^{i}(x)-v^k_m|
    \\
    &\qquad \leq\int_{B(x,2\epsilon)} |g^j(x)-g^j(y)|\, |\partial_{x_{m}}N_{i}(x,y)|\,dy
    \\
    &\qquad\quad+\int_{B(x,2\epsilon)}
    |g^j(y)-g^j(x)| \,|N_{i}(x,y)|\ \eta'\left(\frac{|x-y|}{\epsilon}\right)\frac{|x_{m}-y_{m}|}{|x-y|}\,\frac{dy}{\epsilon}
    \\
    &\qquad\quad+ \int_{B(x,2\epsilon)}|g^j(x)|\,
    |x_{i}-y_{i}|\int_{1}^{\infty}|\partial_{m}\psi(y+\alpha(x-y))| \,
    \alpha(\alpha-1)\,d\alpha\,dy  
    \\
    &=:F + G + H.
\end{aligned}
\end{equation*}
By the definition of modulus of continuity it follows that
\begin{equation*}
F\leq M \int_{ B(x,2\epsilon)}\frac{
  |g^{j}(x)-g^{j}(y)|}{|y-x|^{3}}\,dy \leq 4\pi M \int_{\rho<2\epsilon}\frac{\omega(g,\rho)}{\rho}d\rho.
\end{equation*}
By the Dini continuity of $g$ --and the consequent absolute continuity of the integral of
the modulus of continuity--  the last term vanishes as $\epsilon$ goes to zero,
independently of $x\in\Omega'$.

The second term $G$ is estimated in the same way since by the properties of $\psi$ and
$\eta$ it follows that
\begin{equation*}
  G\leq \int_{\epsilon}^{2\epsilon}\frac{\omega(g,\rho)}{\rho}\,d\rho,
\end{equation*}
and, again by the absolute continuity of the integral, the term $G$ vanishes as $\epsilon$ goes to
zero, independently of $x\in\Omega'$.

Finally, 
by using $\partial_{j}\psi$ instead of $\psi$ as in the previous proofs, it follows that for any $q\,>3$ and suitable
constants $c', c''$
\begin{equation*}
  \begin{aligned}
    |H|&\leq c'\max_{\overline{\Omega}}|g^{j}(x)|\,\left\Vert
      \eta\left(\frac{|x-y|}{\epsilon}\right)-1\right\Vert_{L^{q}(\Omega)}
    \\
    &\leq c''
    \left\|\eta\left(\frac{|x-y|}{\epsilon}\right)-1\right\|_{L^{q}\left(B(x,\,\diam\,\Omega)\right)}
  \end{aligned}
\end{equation*}
Since the latter norm vanishes as $\epsilon$ goes to zero, for any $q>3$, and
independently of $x\in\Omega$, point 2) follows.

By the classical theorem on a converging sequence of functions whose derivatives converge
uniformly, it follows 3), while 4)  follows since $\mathcal{R}^{\epsilon} g\in C^\infty(\R^3)$.
\end{proof}
\begin{remark}
  By a change of variables one can easily consider the case in which $\Omega$ is
  star-shaped with respect to a ball of positive radius $r$, but not necessarily centered
  at the origin.
\end{remark}
\section*{Acknowledgments}
The research that led to the present paper was partially supported by a grant of the group GNAMPA of INdAM.
%
\def\ocirc#1{\ifmmode\setbox0=\hbox{$#1$}\dimen0=\ht0 \advance\dimen0
  by1pt\rlap{\hbox to\wd0{\hss\raise\dimen0
  \hbox{\hskip.2em$\scriptscriptstyle\circ$}\hss}}#1\else {\accent"17 #1}\fi}
  \def\cprime{$'$} \def\polhk#1{\setbox0=\hbox{#1}{\ooalign{\hidewidth
  \lower1.5ex\hbox{`}\hidewidth\crcr\unhbox0}}} \def\cprime{$'$}
\providecommand{\href}[2]{#2}
\providecommand{\arxiv}[1]{\href{http://arxiv.org/abs/#1}{arXiv:#1}}
\providecommand{\url}[1]{\texttt{#1}}
\providecommand{\urlprefix}{URL }

\end{document}